\documentclass[12pt]{amsart}
\usepackage[english]{babel}
\usepackage[T1]{fontenc}
\usepackage[utf8]{inputenc}
\usepackage{indentfirst}
\usepackage{fancyhdr}
\usepackage{amsthm}
\usepackage{amsmath}
\usepackage{amssymb}
\usepackage{braket}
\usepackage{cite,enumitem,graphicx}
\usepackage[colorlinks=true,urlcolor=blue,
citecolor=red,linkcolor=blue,linktocpage,
bookmarksnumbered,bookmarksopen,pdfpagelabels,]{hyperref}

\newtheorem{Th}{Theorem}[section]
\newtheorem{Prop}[Th]{Proposition}
\newtheorem{Lem}[Th]{Lemma}
\newtheorem{Cor}[Th]{Corollary}
\theoremstyle{definition}

\theoremstyle{remark}
\newtheorem{Rem}[Th]{Remark}

\newcommand{\R}{\mathbb{R}}
\newcommand{\rn}{\R^N}
\newcommand{\hrn}{H^1(\rn)}

\newcommand{\cS}{\mathcal{S}}
\newcommand{\cD}{\mathcal{D}}
\newcommand{\cC}{\mathcal{C}}

\newcommand{\cO}{\mathcal{O}}
\newcommand{\cG}{\mathcal{G}}

\numberwithin{equation}{section}

\title[Normalized ground states to an $L^2$-(sub)critical Schr\"odinger system]{Normalized ground states to a cooperative system of Schr\"odinger equations with generic $L^2$-subcritical or $L^2$-critical nonlinearity}

\author[Jacopo Schino]{Jacopo Schino}

\address{Institute of Mathematics of the Polish Academy of Sciences
\newline\indent ul. \'Sniadeckich 8, 00-656 Warsaw, Poland
\newline and
\newline\indent Department of Mathematics, North Carolina State University
\newline\indent 2311 Stinson Dr, 27607 Raleigh, NC, USA}

\email{\href{mailto:jschino@ncsu.edu}{jschino@ncsu.edu}}



\begin{document}
	\maketitle
	
\begin{abstract}
We look for ground state solutions to the Schr\"odinger-type system
\[
\begin{cases}
-\Delta u_j+\lambda_ju_j=\partial_jF(u)\\
\int_{\rn}u_j^2\,dx=a_j^2\\
(\lambda_j,u_j)\in\R\times\hrn
\end{cases}
j\in\{1,\dots,M\}
\]
with $N,M\ge1$, where $a=(a_1,\dots,a_M)\in]0,\infty[^M$ is prescribed and $(\lambda,u)=(\lambda_1,\dots,\lambda_M,u_1,\dots u_M)$ is the unknown. We provide generic assumptions about the nonlinearity $F$ which correspond to the $L^2$-subcritical and $L^2$-critical cases, i.e., when the energy is bounded from below for all or some values of $a$. Making use of a recent idea, we minimize the energy over the constraint $\Set{\left|u_j\right|_{L^2}\le a_j \text{ for all } j}$ and then provide further assumptions that ensure $|u_j|_{L^2}=a_j$.
\end{abstract}

\section{Introduction and statement of the results}

We study the problem
\begin{equation}\label{e-main}
	\begin{cases}
		-\Delta u_j+\lambda_ju_j=\partial_jF(u)\\
		\int_{\rn}u_j^2\,dx=a_j^2\\
		(\lambda_j,u_j)\in\R\times\hrn
	\end{cases}
	j\in\{1,\dots,M\}
\end{equation}
with $N,M\ge1$. Here $a=(a_1,\dots,a_M)\in]0,\infty[^M$ is prescribed, while $(\lambda,u)=(\lambda_1,\dots,\lambda_M,u_1,\dots u_M)$ is the unknown.

Problems as \eqref{e-main} arise when standing waves solutions to Schr\"odinger-type systems
\begin{equation}\label{e-time}
	\begin{cases}
		\mathrm{i}\partial_t\Phi_1-\Delta\Phi_1=\partial_1F(\Phi)\\
		\cdots\\
		\mathrm{i}\partial_t\Phi_M-\Delta\Phi_M=\partial_MF(\Phi)
	\end{cases}
	\Phi=(\Phi_1,\dots,\Phi_M)
\end{equation}
are searched for, i.e., $\Phi_j(t,x)=e^{-\mathrm{i}\lambda_jt}u_j(x)$, with $F(\Phi) = F(|\Phi_1|,\dots,|\Phi_M|)$. Systems as \eqref{e-time} describe natural phenomena in several areas of Physics such as nonlinear optics \cite{Akozbek,SluEgg} and Bose-Einstein condensation \cite{LSSY,PiSt}. In these fields, not only are the masses $\int_{\rn}u_j^2\,dx$ conserved in time (together with the energy, see the functional $J$ below) \cite{Cazenave:book,CazeLions}, but they also have a precise physical meaning, i.e., the power supplies and the total numbers of atoms respectively. It makes therefore sense to introduce the $L^2$ constraints in \eqref{e-main}. With this approach, $\lambda$ in the equation appears as (an $M$-tuple of) Lagrange multipliers; from a physical point of view it represents the chemical potentials of the standing waves.

Under standard assumptions about $F$ it is shown that the solutions to \eqref{e-main} are the critical points of the $\cC^1$ energy functional
\[
J\colon u\in\hrn^M\mapsto\int_{\rn}\frac12|\nabla u|^2-F(u)\,dx\in\R
\]
restricted to the $\cC^1$ manifold
\begin{equation*}
	\cS := \Set{u\in\hrn^M | \int_{\rn }u_j^2 \, dx = a_j^2 \text{ for every } j\in\{1,\dots,M\}}.
\end{equation*}

The case when $J|_{\cS}$ is bounded from below for all $a\in]0,\infty[^M$ is known in the literature as $L^2$-subcritical (or mass-subcritical), while the case when it is unbounded from below for all $a$ is known as $L^2$-supercritical (or mass-supercritical); if such property of being (un)bounded from below depends on $a$, the case is referred to as $L^2$-critical (or mass-critical). When $F$ is of power type, i.e., $F(u)=c|u|^p$ for some $c>0$ and $2<p<2^*$ (let us consider $M=1$ for simplicity), these cases correspond to the exponent $p$ being, respectively, less than, greater than, or equal to the threshold value
\[
2_\#:=2+\frac4N,
\]
known in the literature as the $L^2$-critical (or mass-critical) exponent. As usual, $2^*=2N/(N-2)$ if $N\ge3$, $2^*=\infty$ if $N\in\{1,2\}$. We point out that the value $2_\#$ plays a relevant role also in problems without prescribed $L^2$ norm, e.g., the orbital stability of the ground state solutions to \eqref{e-main} (cf. \cite{CazeLions}).

In this paper, we provide some mild assumptions on the nonlinearity $F$ which correspond to the $L^2$-subcritical and $L^2$-critical cases and ensure the existence of a ground state solution to \eqref{e-main}. If $(\lambda,u)\in\R^M\times\hrn^M$ is a solution to \eqref{e-main}, we call it a \textit{ground state solution} (\textit{ground state} in short) if and only if $u$ minimizes $J$ over the set
\[
\cD := \Set{v\in\hrn^M | \int_{\rn} v_j^2 \, dx \le a_j^2 \text{ for every } j\in\{1,\dots,M\}},
\]
a stronger requirement that the seemingly more natural one that a ground state minimizes $J|_{\cS}$.

Now we list our assumptions about the nonlinearity.
\begin{itemize}
	\item [(F0)] $F\in\cC^1(\R^M)$ and
	\begin{itemize}
		\item if $N=1$, there exists $S>0$ such that $|\nabla F(u)| \le S|u|$ for every $u\in[-1,1]^M$;
		\item if $N=2$, for every $b>0$ there exists $S_b>0$ such that $|\nabla F(u)| \le S_b\bigl(|u| + e^{b|u|^2}-1\bigr)$ for every $u\in\R^M$;
		\item if $N\ge3$, there exists $S>0$ such that $|\nabla F(u)| \le S(|u|+|u|^{2^*-1})$ for every $u\in\R^M$.
	\end{itemize}
	\item [(F1)] $\displaystyle\eta_\infty:=\limsup_{|u|\to\infty}\frac{F(u)}{|u|^{2_\#}}<\infty$.
	\item [(F2)] $\displaystyle\lim_{u\to0}\frac{F(u)}{|u|^2}=0$.
	\item [(F3)] $\displaystyle\eta_0:=\liminf_{u\to0}\frac{F(u)}{|u|^{2_\#}}>0$.
\end{itemize}

In (F3) the case $\eta_0=\infty$ is allowed. When $N\ge5$ and $M\ge2$, we consider the following assumption for a function $f\colon[0,\infty[\to[0,\infty[$, which will be needed for some of our results (in particular, (P) will be required for the derivatives of the additive terms of $F$ restricted to $[0,\infty[$, cf. Theorem \ref{T-main2} below).
\begin{itemize}	
	\item [(P)] There exists $q\le N/(N-2)$ such that $\liminf_{t\to0^+}f(t)/t^q>0$.
\end{itemize}

We will also need hypotheses about the $M$-tuple of radii $a$, i.e.,
\begin{eqnarray}
	\label{e-etas} 2\eta_\infty C_{N,2_\#}^{2_\#}|a|^{4/N}<1,\\
	\label{e-etal} 2\eta_0C_{N,2_\#}^{2_\#}M^{2/N}\min_{1\le j\le M}a_j^{4/N}>1,
\end{eqnarray}
where $C_{N,2_\#} > 0$ is defined in \eqref{e-GN}. Of course, if $\eta_\infty=0$ (resp. $\eta_0=\infty)$, then \eqref{e-etas} (resp. \eqref{e-etal}) is automatically satisfied.

The problem of finding reasonable assumptions for the existence of solutions to \eqref{e-main} in the $L^2$-subcritical (sometimes $L^2$-critical) case dates back to the work of T. Cazenave \& P.-L. Lions \cite{CazeLions} and C. A. Stuart \cite{Stuart82}, for $M=1$, and P.-L. Lions \cite{Lions84_2}, for $M\ge1$. More recently, it was dealt with by L. Jeanjean \& S.-S. Lu \cite{JeanLu} for $M=1$, who considered the issue of (infinitely many) nonradial solutions, and by T. Bartsch \& L. Jeanjean \cite{BarJean} for $M=2$, but only with explicit power type nonlinearities and without considering the $L^2$-critical case. Still concerning recent times, the relative compactness of minimizing sequences, which is closely connected with the orbital stability, has been given some attention too, e.g., in \cite{Ikoma,Shibata_2014} for $M=1$ and \cite{GouJean,Shibata_2017} for $M\ge2$. See also the references therein for a more complete bibliography. Most of these recent papers rely on some features of the ground state energy map or Palais-Smale sequences with additional properties in the spirit of \cite{BerLionsII}, which makes the arguments more involved than the one used here. In both cases, one of the main struggles consists in proving that limit points of weakly convergent sequences maintain the $L^2$ norm because the embedding $\hrn\hookrightarrow L^2(\rn)$ is \textit{not} compact, not even when one restricts to radially symmetric functions.

Due to this lack of compactness, we do not know, in general, if a bounded sequence in $H^1(\rn)$ converges strongly in $L^2(\rn)$ up to a subsequence. In particular, the limit of a weakly convergent sequence in $\cS$ need not belong to $\cS$. For this reason we work with the set $\cD$ and, clearly, the limit of a weakly convergent sequence in $\cD$ belongs to $\cD$. Then it is easier to prove that such a weak limit point is a minimizer, which gives additional information to use when we prove that, in fact, it belongs to $\cS$: this is more delicate when $M\ge2$, basically because $\cS \subsetneq \partial \cD$, in contrast with the case $M=1$, where $\cS = \partial\cD$. Therefore, a second aim of this paper is to explore how choosing $\cD$ over $\cS$ can improve the previous work on this topic.

The idea of working with the set $\cD$ was introduced in \cite{BiegMed} for a single equation in the $L^2$-critical or -supercritical (and Sobolev-subcritical) case and then used in \cite{MedSc} for a system of equations in a broader regime, i.e., allowing Sobolev-critical nonlinearities. More in details, since in that setting the energy functional is unbounded from below, the authors make use of a natural constraint of Nehari-Poho\v{z}aev type that allows to recover such boundedness. In addition, in \cite{MedSc} a second advantage of working with $\cD$ is exploited, i.e., the Lagrange multipliers coming from this constraint are nonnegative because associated with a minimizer. This property, which is used here as well and is based on a result of Clarke's \cite{Clarke}, is helpful when one has no other ways to obtain information about the Lagrange multipliers.

Before stating our results, we recall the best constant $C_{N,p} > 0$ in the Gagliardo--Nirenberg inequality
\begin{equation}\label{e-GN}
	|u|_p\le C_{N,p}|u|_2^{1-\delta_p}|\nabla u|_2^{\delta_p} \text{ for every } u\in\hrn
\end{equation}
with $2<p<2^*$ and $\delta_p=N\left(\frac12-\frac1p\right)$. Note that
\[
p\delta_p=N\left(\frac{p}{2}-1\right)\begin{cases}
	<2 \quad \text{ if } 2<p<2_\#\\
	=2 \quad \text{ if } p=2_\#\\
	>2 \quad \text{ if } 2_\#<p<2^*
\end{cases}
\]
and $0<\delta_p<1$ because $2<p<2^*$.

\begin{Th}\label{T-main1}
	If $M=1$ and (F0)--(F3), \eqref{e-etas}, and \eqref{e-etal} hold, then there exists a solution $(\lambda,u)\in]0,\infty[\times\cS$ to \eqref{e-main} such that $0 > J(u) = \inf_{\cD}J = \inf_{\cS}J$.\\
	If, moreover, $N\ge2$, $F'$ is locally Lipschitz continuous, or $F$ is even, then every minimizer of $J|_{\cD}$ has constant sign and, up to a translation, is radial and radially monotone.
\end{Th}

Theorem \ref{T-main1} refines \cite[Theorem 1.1 (ii)]{JeanLu2} in the sense that \eqref{e-etal} gives an explicit (in terms of the Gagliardo--Nirenberg constant $C_{N,2_\#}$) lower bound on $a$ for solutions to exist; on the other hand, we do not know whether this bound is optimal.

The more one assumes about $F$, the more can be said about $u$. In particular, we have as follows.
\begin{Prop}\label{P-decreasing}
	Let the assumptions of Theorem \ref{T-main1} hold, with $N\ge2$, $F'$ locally Lipschitz continuous, or $F$ even, and let $(\lambda,u)$ be given therein. If $F$ is nondecreasing on $[0,\infty[$ and nonincreasing on $[-\infty,0]$, then $|u|>0$. If, moreover, there exist $t_0,t^0>0$ such that $F'(t)\le\lambda t$ for every $t\in[0,t_0]$, $F'(t)>\lambda t$ for every $t>t_0$, $F'(t) \ge \lambda t$ for every $t\in[-t^0,0]$, and $F'(t)<\lambda t$ for every $t<-t^0$, then $u$ is radially strictly monotone.
\end{Prop}

\begin{Rem}
	Solutions $(\lambda,u)$ to \eqref{e-main} where $u$ is nonnegative (resp. nonpositive) exist also if, in addition to the assumptions of the first part of Theorem \ref{T-main1}, $F'(t) \ge 0$ for all $t\le0$ (resp. $F'(t) \le 0$ for all $t\ge0$) because, denoting by $u_- := \max\{-u,0\}$ the negative part of $u$ (resp. by $u_+ := \max\{u,0\}$ the positive part of $u$),
	\[\begin{split}
		0 \ge -|\nabla u_-|_2^2 - \lambda |u_-|_2^2 & = - \int_{\rn} |\nabla u_-|^2 + \lambda u_-^2 \, dx = \int_{\rn} \nabla u \cdot \nabla u_- + \lambda uu_- \, dx\\
		& = \int_{\rn} F'(u)u_- \, dx = \int_{\rn} F'(-u_-)u_- \, dx \ge 0
	\end{split}\]
	(resp. $0 \le |\nabla u_+|_2^2 + \lambda |u_+|_2^2 \le 0$), although there is no information about the symmetry of $u$. In this case, we still have that $|u|>0$ if $F$ is nondecreasing on $[0,\infty[$ (resp. nonincreasing on $]-\infty,0]$).
\end{Rem}

Concerning the existence of nonradial solutions, we have the following result. When $N\ge4$, fix $2\le K\le N/2$ and consider $\tau\in\cO(N)$ defined by $$\tau(x_1,x_2,x_3)=(x_2,x_1,x_3)$$ for every $x=(x_1,x_2,x_3)\in\R^K\times\R^K\times\R^{N-2K}=\rn$. Define
\[
H_\tau:=\Set{u\in\hrn|u=-u(\tau\cdot)}
\]
and note that the only radial function contained in $H_\tau$ is the trivial one. Finally, let
\[
X:=\Set{u\in H_\tau|u=u(g\cdot) \text{ for every } g\in\cO(K)\times\cO(K)\times \{I_{N-2K}\}},
\]
which a fortiori does \textit{not} contain any nontrivial radial functions. When $2K=N$, we agree that the component $x_3$ in the definition of $\tau$ and the identity matrix $I_{N-2K}$ in the definition of $X$ do not appear.

\begin{Prop}\label{P-main}
	If $N\ge4$, $M=1$, $F$ is even, (F0)--(F3) and \eqref{e-etas} hold, and $\eta_0=\infty$, then there exists a solution $(\mu,v) \in ]0,\infty[ \times (\cS\cap X)$ to \eqref{e-main} such that $0 > J(v) = \inf_{\cD \cap X}J = \inf_{\cS \cap X}J > \inf_{\cD}J = \inf_{\cS}J$.
\end{Prop}

Note that under the assumptions of Proposition \ref{P-main} there exist two distinct solutions to \eqref{e-main}: $(\lambda,u) \in ]0,\infty[ \times \cS$, which is also radial (up to a translation) and a ground state, and $(\mu,v) \in ]0,\infty[ \times (\cS \cap X)$. We remark that L. Jeanjean \& S.-S. Lu \cite{JeanLu} obtained a solution in $X$ for sufficiently large $a$ with $\eta_0<\infty$, but assuming $\eta_\infty=0$.

When $M\ge2$, the following holds.

\begin{Th}\label{T-main2}
	Assume that $M\ge2$, (F0)--(F3) and \eqref{e-etas} hold, and for every $j\in\{1,\dots,M\}$ and every $\alpha,\beta$ in the proper range $F_j,\widetilde{F}_{\alpha,\beta} \in \cC^1(\R) \setminus \{0\}$ are even, nonnegative, nondecreasing on $[0,\infty[$, and $\widetilde{F}_{j,k}(0)=0$.
	\begin{itemize}
		\item [(a)] If \eqref{e-etal} holds, $L\ge1$ is an integer, $F$ is of the form
		\[
		F(u) = \sum_{j=1}^M F_j(u_j) + \sum_{\ell=1}^L \prod_{j=1}^{M} \widetilde{F}_{\ell,j}(u_j),
		\]
		and $N\le4$ or each $F_j'|_{[0,\infty[}$ satisfies (P), then there exists $(\lambda,u) \in ]0,\infty[^M \times \cS$ to \eqref{e-main} such that $0 > J(u) = \inf_{\cD}J = \inf_{\cS}J$ and each component of $u$ is radial, positive, and radially nonincreasing.
		\item [(b)] If $\eta_0=\infty$ and $F$ is of the form
		\[
		F(u) = \sum_{j=1}^M F_j(u_j) + \sum_{i<j} \widetilde{F}_{i,j}(u_i) \widetilde{F}_{j,i}(u_j)
		\]
		with $\widetilde{F}_{i,j}$ positive on $\R \setminus \{0\}$ for every $i\ne j$, then there exists $(\lambda,u) \in [0,\infty[^M \times \cS$ to \eqref{e-main} such that $0 > J(u) = \inf_{\cD}J = \inf_{\cS}J$, $\max_{j=1,\dots,M}\lambda_j>0$, and each component of $u$ is radial, positive, and radially nonincreasing.\\
		If, moreover, $N\le4$ or each $F_j'|_{[0,\infty[}$ satisfies (P), then $\lambda \in ]0,\infty[^M$.
	\end{itemize}
\end{Th}

Observe that the coupling terms of $F$ in Theorem \ref{T-main2} are nonnegative, i.e., the system is cooperative.

It is easy to check that a necessary condition for \eqref{e-etas} and \eqref{e-etal} to hold simultaneously is that $\eta_0>\eta_\infty$. This is what holds in the $L^2$-subcritical case, where $\eta_0=\infty$ and $\eta_\infty=0$. At the same time it rules out the $L^2$-critical case when $F$ is of power type, i.e., $F(u)=|u|^{2_\#}/2_\#$ when $M=1$ and similarly when $M\ge2$. This reflects the fact that the $L^2$-critical regime is indeed a delicate case. Nevertheless, if \eqref{e-etas} and \eqref{e-etal} both hold, then under additional conditions we can find a ground state solution to \eqref{e-main} for uncountably many values of $a$ even when the behaviour of $F$ is $L^2$-critical both at zero and at infinity. Observe also that $\eta_0=\infty$ is a necessary condition for Theorem \ref{T-main2} (a) to hold if $N\ge5$.

The reason why we consider two possible forms of $F$ in Theorem \ref{T-main2} is due to the different proofs we provide, which rely on particular features of such two nonlinearities: roughly speaking, we need the coupling part to vanish whenever $u_j=0$ for some $j\in\{1,\dots,M\}$ in point (a), while on the contrary, we need it \textit{not} to vanish whenever $u_i,u_j\ne0$ for some $i\ne j$ in point (b).

\begin{Rem}
	Theorem \ref{T-main2} (b) still holds with a richer coupling part. For instance, one can add a second term $\sum_{i<j} \widetilde{G}_{i,j}(u_i) \widetilde{G}_{j,i}(u_j)$ or coupling terms as in point (a). We decided not to focus on the most generic possible form in order to simplify notations.
\end{Rem}

Our final result deals with the orbital stability of the ground state solutions to \eqref{e-main}. Define $\cG := \Set{ u \in \cD | J(u) = \inf_{\cD}J }$.
\begin{Prop}\label{P-os}
	In the assumptions of the first part of Theorem \ref{T-main1} or the first part of Theorem \ref{T-main2} (b), suppose additionally the following.
	\begin{itemize}
		\item If $M=1$, then $F$ is even.
		\item There exist $S_1 > 0$ and $p \in ]2,2^*[$ such that $|\nabla F(u) - \nabla F(v)| \le S_1 (1 + |u| + |v|)^{p-2} |u - v|$ for every $u,v \in \R^M$.
		\item There exist $S_2 > 0$ and $\sigma \in ]2,2_\#[$ such that $F(u) \le S_2 (|u|^2 + |u|^\sigma)$ for every $u \in \R^M$.
	\end{itemize}
	Then for every $\varepsilon>0$ there exists $\delta>0$ such that for every $\phi \in \hrn^M$ satisfying
	\[
	\inf_{u\in\cG}\|\phi-u\| \le \delta
	\]
	there holds
	\[
	\sup_{t\ge0}\inf_{u\in\cG}\|\Phi(t,\cdot)-u\| \le \varepsilon,
	\]
	where $\Phi$ is the solution to the Cauchy problem consisting of \eqref{e-time} and the initial datum $\Phi(0,\cdot) = \phi$, and $\|\cdot\|$ is the norm of $\hrn^M$.
\end{Prop}
The proof is classical (see, e.g., \cite{CazeLions,GouJean}) and based on the relative compactness up to translations of minimizing sequences -- cf. Corollary \ref{C-cpt}, the global well-posedness of the Cauchy problem, and the conservations of mass and energy, therefore we omit it. We remark that the evenness of $F$ when $M=1$ is used to infer the aforementioned conservations (when $M\ge2$, we need $F$ to be even in each component, which is already part of the assumptions). We also point out that, as observed in \cite[page 222]{Shibata_2014}, the other two additional assumptions are used for the global well-posedness (see, e.g., \cite{Cazenave:book}).

We conclude providing some examples for $F$ in Theorems \ref{T-main1} and \ref{T-main2} and Propositions \ref{P-decreasing} and \ref{P-main}, beginning with the case $M=1$. A first model for the nonlinearity is
\[
F(u)=\frac{\nu}{2_\#}|u|^{2_\#}+\frac{\bar\nu}{p}|u|^p
\]
for some $\nu\ge0$, $\bar\nu>0$, and $2<p<2_\#$, in which case one has $\eta_0=\infty$ and $\eta_\infty=\nu/2_\#$. A second model is a sort of counterpart of the first one, i.e.,
\begin{equation}\label{e-1}
	F(u)=\int_0^{|u|}\min\{t^{2_\#},t^p\}\,dt
\end{equation}
with $2<p<2_\#$, in which case one has $\eta_0=1/2_\#$ and $\eta_\infty=0$.

Now define $F^*\colon[0,\infty[\to\R$ by $F^*(0)=0$ and
\[
F^*(t)=
\begin{cases}
	-\frac{t^2}{\ln t} \quad & \text{if } 0<t<\frac12\\
	\frac{b}2\bigl((b+2)t-\frac12-\frac{b}2\bigr) \quad & \text{if } \frac12\le t\le1\\
	-t^2+2ct-1-\frac{b}4(b+1) \quad & \text{if } 1<t\le c\\
	F^*(2c-t) \quad & \text{if } c<t\le2c\\
	0 \quad & \text{if } t>2c
\end{cases}
\]
with $b=1/\ln2$ and $c=b(b+2)/4+1$ and let $F(u):=F^*(|u|)$. Then there holds $\eta_0=\infty$ and $\eta_\infty=0$. Note that (F2) is satisfied, but $\lim_{t\to0}F(t)/|t|^p=\infty$ for every $p>2$. One can also modify the previous example in order to have $\eta_0<\infty$ by defining
\[
F^*(t)=
\begin{cases}
	\frac{t^{2_\#}}{2_\#} \quad & \text{if } 0\le t\le1\\
	t-1+\frac1{2_\#} \quad & \text{if } 1<t<2\\
	-t^2+5t-5+\frac1{2_\#} \quad & \text{if } 2\le t\le\frac52\\
	F^*(5-t) \quad & \text{if } \frac52<t<5\\
	0 \quad & \text{if } t\ge5.
\end{cases}
\]
Notice that, in both examples, $F^*$ is \textit{not} monotone, therefore such examples do not suit the case $M\ge2$; however, if we modify $F^*$ such that it is constant after it reaches its maximum, then we can consider it also for that case.

For the case when both $\eta_\infty$ and $\eta_0$ are finite and positive, an example is
\begin{equation}\label{e-2}
	F(u)=
	\begin{cases}
		\frac{|u|^{2_\#}}{2_\#} \quad & \text{if } 0\le|u|\le1\\
		|u|-1+\frac1{2_\#} \quad & \text{if } 1<|u|<2\\
		\frac{|u|^{2_\#}}{2^{2_\#-1}2_\#}+1-\frac1{2_\#} \quad & \text{if } |u|\ge2.
	\end{cases}
\end{equation}

Finally, sign-changing nonlinearities can be obtained, e.g., by adding the term $-|u|^q/q$ in the previous examples, where $2_\# < q \le 2^*$ if $N\ge3$, $2_\# < q < 2^*$ otherwise.

Concerning the case $M\ge2$, similarly a model for the nonlinearity (let us begin with the assumptions of Theorem \ref{T-main2} (a)) is
\begin{equation}\label{e-exM}
	F(u) = \sum_{j=1}^M \left( \frac{\nu_j}{2_\#} |u_j|^{2_\#} + \frac{\bar\nu_j}{p_j} |u_j|^{p_j} \right) + \alpha\prod_{j=1}^M |u_j|^{r_j} + \beta\prod_{j=1}^M |u_j|^{\bar r_j}
\end{equation}
for some $\nu_j,\alpha,\beta\ge0$, $\bar\nu_j>0$, $r_j,\bar r_j>1$, and $2<p_j<2_\#$ such that $\alpha+\beta>0$, $\sum_{j=1}^Mr_j=2_\#$, and $\sum_{j=1}^M\bar r_j<2_\#$. When $N\ge5$ (which implies $M=2$, cf. Remark \ref{R-K}), we need to add the term
\[
\frac{\tilde\nu_1}{q_1}|u_1|^{q_1}+\frac{\tilde\nu_2}{q_2}|u_2|^{q_2}, \quad 2<q_1,q_2\le\frac{2N-2}{N-2}, \, \tilde\nu_1,\tilde\nu_2>0,
\] (then we can allow $\bar\nu_j=0$). In this case, again one has $\eta_0=\infty$. If $\alpha=0$, then $\eta_\infty=\max_{j=1,\dots,M}\nu_j/2_\#$; if $M=2$ and $\nu_1=\nu_2=0$, then $\eta_\infty=\alpha\sqrt{r_1^{r_1}r_2^{r_2}/2_\#^{2_\#}}$ (see Appendix \ref{App} for more details on such computations).

As for the assumptions of Theorem \ref{T-main2} (b) (let us consider the case $M=3$ for simplicity), likewise we can take as a model
\begin{equation}\label{e-exM2}\begin{split}
		F(u) & = \sum_{j=1}^3 \left( \frac{\nu_j}{2_\#} |u_j|^{2_\#} + \frac{\bar\nu_j}{p_j} |u_j|^{p_j} \right) + \alpha\prod_{j=1}^3 |u_j|^{r_j}\\
		& + \alpha_{12}|u_1|^{r_{12}}|u_2|^{r_{21}} + \alpha_{13}|u_1|^{r_{13}}|u_3|^{r_{31}} + \alpha_{23}|u_2|^{r_{23}}|u_3|^{r_{32}}\\
		& + \beta\prod_{j=1}^3 |u_j|^{\bar{r}_j} + \beta_{12}|u_1|^{\bar{r}_{12}}|u_2|^{\bar{r}_{21}} + \beta_{13}|u_1|^{\bar{r}_{13}}|u_3|^{\bar{r}_{31}} + \beta_{23}|u_2|^{\bar{r}_{23}}|u_3|^{\bar{r}_{32}}
\end{split}\end{equation}
for some $\alpha,\beta \ge 0$, $\alpha_{ij},\beta_{ij} > 0$, $r_{ij},\bar{r}_{ij} > 1$ such that $r_{ij} + r_{ji} = 2_\#$ and $\bar{r}_{ij} + \bar{r}_{ji} < 2_\#$, and $\nu_j,\bar{\nu}_j,p_j,r_j,\bar{r}_j$ as before.

Finally, one can take $F_j$ and $\widetilde{F}_{j,k}$ as in \eqref{e-1} or \eqref{e-2}, possibly with additional restrictions on $F_j$ in a similar way as before if $N\ge5$.

\begin{Rem}\label{R-K}
	Although there are no explicit restrictions on $M$ in Theorem \ref{T-main2}, the example in \eqref{e-exM} shows that we could need $M$ not to be too large. As a matter of fact, since $r_j>1$ for every $j\in\{1,\dots,M\}$, there holds
	\[
	M < \sum_{j=1}^M r_j = 2_\#
	\]
	(or likewise with $\bar{r}_j$ if $\alpha=0$). As for examples of when $M$ can be arbitrary, one can replace each $|\cdot|^{r_j}$ and $|\cdot|^{\bar{r}_j}$ in \eqref{e-exM} with suitable bounded functions (cf. the paragraph before \eqref{e-2}) or take $\alpha=\beta=0$ in \eqref{e-exM2}.
\end{Rem}

\subsection{Notations and structure of the paper}
The interior of a set $A$ will be denoted by $\mathring{A}$ and its characteristic function by $\chi_A$; if $A$ is a Borel set, then $|A|$ stands for its Lebesgue measure. $B_r$ is the open ball of radius $r\ge0$ centred at $0$, while the $L^p(\rn)$ norm of a function $w$ will be denoted by $|w|_p$, $1\le p\le\infty$. Concerning sequences, we will write $e_n$ for $(e_n)_{n\in\mathbb{N}}$ and $e_n\in E$ for $(e_n)_{n\in\mathbb{N}}\subset E$. If $e_n$ represents a sequence of $M$-tuples of numbers or functions, then we will write $e^n$ and use the subscript for the various components of the $M$-tuple. Sometimes we will write explicitly, for $a=(a_1,\dots,a_M)\in]0,\infty[^M$, $\cS(a)$ instead of $\cS$, $\cD(a)$ instead of $\cD$, and $m(a) := \inf\Set{J(u)|u\in\cD(a)}$. $F_+$ stands for the positive part of $F$.

In Section \ref{M1} we study the case $M=1$, in Section \ref{M2} we study the case $M\ge2$, and in Section \ref{gsem} we provide some results on the ground state energy map $m$. Appendix \ref{App} contains explicit computations about the examples for the nonlinearity provided above.

\section{Ground states for $M=1$}\label{M1}

In this and the following sections, we will always assume that (F0) is satisfied and we will make use of it without explicit mention.

\begin{Lem}\label{L-cbb}
	If (F0)--(F2) and \eqref{e-etas} hold, then $J|_{\cD}$ is coercive and bounded from below.
\end{Lem}
\begin{proof}
	From (F1) and (F2), for every $\varepsilon>0$ there exists $c_\varepsilon>0$ such that $F(u)\le c_\varepsilon u^2+(\varepsilon+\eta_\infty)|u|^{2_\#}$ for every $u\in\R$. In view of \eqref{e-GN}, for every $u\in\cD$ we have
	\[\begin{split}
		J(u)&\ge\frac12|\nabla u|_2^2-c_\varepsilon |u|_2^2-(\varepsilon+\eta_\infty)|u|_{2_\#}^{2_\#}\\
		&\ge\frac12|\nabla u|_2^2-c_\varepsilon a^2-(\varepsilon+\eta_\infty)a^{4/N}C_{N,2_\#}^{2_\#}|\nabla u|_2^2\\
		&=\left(\frac12-(\varepsilon+\eta_\infty) a^{4/N}C_{N,2_\#}^{2_\#}\right)|\nabla u|_2^2-c_\varepsilon a^2,
	\end{split}\]
	hence the statement holds true for sufficiently small $\varepsilon$.
\end{proof}

\begin{Rem}\label{R-cbb}
	Lemma \ref{L-cbb} still holds for $M\ge2$ because $\left||u|\right|_r=\left|u\right|_r$, $1\le r\le\infty$, and $\left|\nabla|u|\right|_2\le\left|\nabla u\right|_2$, hence one can use \eqref{e-GN} with $|u|$.
\end{Rem}

For $u\in\hrn\setminus\{0\}$ and $s>0$ let $s\star u(x):=s^{N/2}u(sx)$. Note that $|u|_2=|s\star u|_2$.

\begin{Lem}\label{L-neg}
	If (F0), (F3), and \eqref{e-etal} hold, then $\inf_{\cD}J<0$.
\end{Lem}
\begin{proof}
	Fix $u\in L^\infty(\rn)\cap\cD\setminus\{0\}$ and note that
	\[
	J(s\star u)=s^2\int_{\rn}\frac12|\nabla u|^2-\frac{F(s^{N/2}u)}{(s^{N/2})^{2_\#}}\,dx.
	\]
	From (F3), $F(s^{N/2}u) \ge 0$ for sufficiently small $s>0$, hence we can use Fatou's lemma in what follows. If $\eta_0=\infty$, then $\displaystyle\lim_{s\to0^+}\int_{\rn}\frac{F(s^{N/2}u)}{(s^{N/2})^{2_\#}}\,dx=\infty$. If $\eta_0<\infty$, then
	\[
	\limsup_{s\to0^+}\int_{\rn}\frac12|\nabla u|^2-\frac{F(s^{N/2}u)}{(s^{N/2})^{2_\#}}\,dx\le\int_{\rn}\frac12|\nabla u|^2-\eta_0|u|^{2_\#}\,dx,
	\]
	hence the statement holds true if $u\in\cD$ and
	\begin{equation}\label{e-ineqeta0}
		\frac12\int_{\rn}|\nabla u|^2\,dx<\eta_0\int_{\rn}|u|^{2_\#}\,dx.
	\end{equation}
	Let $w\in H^1(\rn)$ be the unique positive radial solution to $-\Delta v+\frac2Nv=v^{2_\#-1}$ in $\rn$ \cite{Kwong}. Then $w\in L^\infty(\rn)$ from standard regularity theory and, moreover, it is such that $$|w|_2^{4/N}=\frac{2_\#}{2C_{N,2_\#}^{2_\#}}$$ and equality holds in \eqref{e-GN} \cite{Weinstein}. If we define $u(x):=w(tx)$ for some $t>0$, then $u\in\cD$ and \eqref{e-ineqeta0} become, respectively,
	\begin{equation}\label{e-ineqs}
		|w|_2^2\le a^2t^N \quad \text{and} \quad t^2|\nabla w|_2^2<2\eta_0|w|_{2_\#}^{2_\#}.
	\end{equation}
	Now, with the help of the properties of $w$, it is easy to check by direct computations that \eqref{e-ineqs} holds if and only if $\displaystyle t\in\left[\frac{\sqrt{1+2/N}}{a^{2/N}C_{N,2_\#}^{1+2/N}},\sqrt{2_\#\eta_0}\right[$.
\end{proof}

\begin{Rem}\label{R-neg}
	(i) When $\eta_0<\infty$, we find a \textit{radial} function $u\in\cD$ such that $J(u)<0$; this is why, if we are interested in \textit{nonradial} solutions at a negative energy level, we need to assume that $\eta_0=\infty$. Concerning how to build a nonzero function $u \in L^\infty(\rn) \cap \cD \cap X$, simply take $0 \ne \tilde{u} \in L^\infty(\rn) \cap \cD$ radial and define $u(x) := \tilde{u}(x) \chi(|x_1|-|x_2|)$, where $\chi \colon \R \to [-1,1]$ is a smooth odd function such that $\chi(t) = 1$ for every $t \ge 1$, cf. \cite[Remark 4.2]{NonradMed}.
	
	(ii) Lemma \ref{L-neg} still holds for $M\ge2$ because, if we set $W\in H^1(\rn)^M$ as $W=(w/\sqrt{M},\dots,w/\sqrt{M})$, then $|W|_{2_\#}=|w|_{2_\#}$, $|\nabla W|_2=|\nabla w|_2$, and $|W_j|_2 = |w|_2/\sqrt{M}$, therefore we can define $u(x):=W(tx)$ and conclude likewise. In addition, it still holds even when $a\in[0,\infty[^M\setminus\{0\}^M$ because, for every $j=1,\dots,M$, one can take $W_j=0$ if $a_j=0$ and $W_j=w/\sqrt{M_*}$ if $a_j>0$, where $M_*\in\{1,\dots,M\}$ is the number of nonzero components of $a$.
\end{Rem}

\begin{Lem}\label{L-subadd}
	Let $a,b>0$ and assume that (F0)--(F2) are satisfied.
	\begin{itemize}
		\item [(i)] $m(\sqrt{a^2+b^2}) \le m(a) + m(b)$.
		\item [(ii)] If $m(a)$ or $m(b)$ are attained at a nontrivial function, then $m(\sqrt{a^2+b^2}) < m(a) + m(b)$.
	\end{itemize}
\end{Lem}
\begin{proof}
	\textit{(i)} Let $\varepsilon>0$. There exists $u \in \cD(a) \cap \cC_c^\infty(\rn)$ and $v \in \cD(b) \cap \cC_c^\infty(\rn)$ such that $J(u) \le m(a) + \varepsilon$ and $J(v) \le m(b) + \varepsilon$. We can assume that the supports of $u$ and $v$ are disjoint, whence $J(u+v) = J(u) + J(v)$ and $|u+v|_2^2 = |u|_2^2 + |v|_2^2 \le a^2 + b^2$. There follows that
	\[
	m(\sqrt{a^2+b^2}) \le J(u+v) \le m(a) + m(b) + 2\varepsilon.
	\]
	\textit{(ii)} We can assume that $m(a)$ is attained, so let $u \in \cD(a) \setminus \{0\}$ such that $J(u) = m(a)$. We want to prove preliminarily that
	\begin{equation}\label{e-m_est}
		\begin{split}
			m(a) \le \frac{a^2}{b^2}m(b) & \quad \text{if } a \ge b,\\
			m(\sqrt{a^2+b^2}) \le \frac{a^2+b^2}{b^2}m(b) & \quad \text{if } a < b.
		\end{split}
	\end{equation}
	Let $\varepsilon>0$ and $v\in\cD(b)$ such that
	\[
	J(v) \le m(b) + \frac{b^2}{a^2+b^2}\varepsilon < m(b) + \frac{b^2}{a^2}\varepsilon.
	\]
	For every $s \ge 1$ there holds
	\begin{equation}\label{e-ms}\begin{split}
			m(\sqrt{s}b) & \le J\bigl(v(\cdot/s^{1/N})\bigr) = s\left(\frac{1}{2s^{2/N}}|\nabla v|_2^2-\int_{\rn}F(v)\,dx\right) \le sJ(v)\\
			& \le s\left(m(b) + \frac{b^2}{a^2+b^2}\varepsilon\right) < s\left(m(b) + \frac{b^2}{a^2}\varepsilon\right).
	\end{split}\end{equation}
	If $a \ge b$, then \eqref{e-ms} yields
	\[
	m(a) = m\left(\frac{a}{b}b\right) < \frac{a^2}{b^2}\left(m(b) + \frac{b^2}{a^2}\varepsilon\right) = \frac{a^2}{b^2}m(b) + \varepsilon,
	\]
	while if $a < b$, then \eqref{e-ms} yields
	\[
	m(\sqrt{a^2+b^2}) = m\left(\sqrt{\frac{a^2+b^2}{b^2}}b\right) \le \frac{a^2+b^2}{b^2}\left(m(b) + \frac{b^2}{a^2+b^2}\varepsilon\right) = \frac{a^2+b^2}{b^2}m(b) + \varepsilon
	\]
	and \eqref{e-m_est} follows from the arbitrariness of $\varepsilon$. Moreover, arguing as in \eqref{e-ms} we observe that $m(\sqrt{s}a) < sm(a)$ for every $s>1$. Using \eqref{e-m_est}, if $a \ge b$, then
	\[
	m(\sqrt{a^2+b^2}) < \frac{a^2+b^2}{a^2}m(a) = m(a) + \frac{b^2}{a^2}m(a) \le m(a) + m(b),
	\]
	while if $a < b$, then
	\[
	m(\sqrt{a^2+b^2}) \le \frac{a^2+b^2}{b^2}m(b) = \frac{a^2}{b^2}m(b) + m(b) < m(a) + m(b).\qedhere
	\]
\end{proof}

\begin{Rem}\label{R-subadd}
	\textit{(i)} The proof of Lemma \ref{L-subadd} (ii) simplifies if both $m(a)$ and $m(b)$ are attained at nontrivial functions because then one has $m(\sqrt{s}b) < sm(b)$ for every $s>1$ as well.
	
	\textit{(ii)} Lemma \ref{L-subadd} (i) still holds for $M\ge2$ and with $a,b\in[0,\infty[^M$, in which case by $\sqrt{a^2+b^2}$ we mean the $M$-tuple $\left(\sqrt{a_1^2+b_1^2},\dots,\sqrt{a_M^2+b_M^2}\right)$.
\end{Rem}

\begin{Lem}\label{L-cptmin}
	If (F0)--(F3), \eqref{e-etas}, and \eqref{e-etal} hold, then every minimizing sequence for $J|_{\cD}$ is relatively compact in $L^p(\rn)$ up to translations, $2<p<2^*$.
\end{Lem}
\begin{proof}
	Let $u_n\in\cD$ such that $\lim_nJ(u_n) = \inf_{\cD}J$; it is bounded due to Lemma \ref{L-cbb}. If
	\[
	\lim_n \max_{y\in\rn} \int_{B(y,1)} u_n^2 \, dx = 0,
	\]
	then in view of \cite[Lemma I.1]{Lions84_2} we obtain that $u_n \to 0$ in $L^{2_\#}(\rn)$, whence $\int_{\rn}F_+(u_n)\,dx \to 0$ from (F1) and (F2). Moreover, up to a subsequence, $u_n\rightharpoonup0$ in $\hrn$ and $u_n\to0$ a.e. in $\rn$, hence $\liminf_nJ(u_n) \ge 0$, a contradiction with Lemma \ref{L-neg}. Then there exist $u\in\cD\setminus\{0\}$ and $y_n\in\rn$ such that, up to a subsequence, $v_n := u_n(\cdot-y_n) - u \rightharpoonup 0$ in $H^1(\rn)$ and $v_n \to 0$ a.e. in $\rn$. From \cite[Theorem 1]{BrezisLieb} we get
	\[
	\lim_nJ(u_n) - J(v_n) = J(u).
	\]
	If
	\[
	\lim_n \max_{y\in\rn} \int_{B(y,1)} v_n^2 \, dx = 0,
	\]
	then again from \cite[Lemma I.1]{Lions84_2} the statement holds true, thus we assume by contradiction that it is not the case. Consequently, as before, there exist $v\in H^1(\rn)\setminus\{0\}$ and $z_n\in\rn$ such that, up to a subsequence, $w_n := v_n(\cdot-z_n) - v \rightharpoonup 0$ in $H^1(\rn)$, $w_n \to 0$ a.e. in $\rn$,
	\[
	\lim_nJ(v_n) - J(w_n) = J(v), \quad \text{and} \quad \lim_n |u_n|_2^2 - |w_n|_2^2 = |u|_2^2 + |v|_2^2
	\]
	(the last equality is again due to \cite[Theorem 1]{BrezisLieb}). If we denote $b := |u|_2 > 0$ and $c := |v|_2 > 0$, then
	\[
	a^2 - b^2 - c^2 \ge \liminf_n |u_n|_2^2 - |u|_2^2 - |v|_2^2 = \liminf_n |w_n|_2^2 =: d^2 \ge 0.
	\]
	If $d>0$, then let $\tilde{w}_n := \frac{d}{|w_n|_2}w_n \in \cS(d)$. From \cite[Lemma 2.4]{Shibata_2014} we have $\lim_n J(w_n) - J(\tilde{w}_n) = 0$, while from Lemma \ref{L-subadd} (i) and Proposition \ref{P-gsem} (i) we obtain
	\[\begin{split}
		m(a) & = \lim_nJ(u_n) = J(u) + J(v) + \lim_nJ(w_n) = J(u) + J(v) + \lim_nJ(\tilde{w}_n)\\
		& \ge m(b) + m(c) + m(d) \ge m(\sqrt{b^2+c^2+d^2}) \ge m(a),
	\end{split}\]
	hence all the inequalities above are in fact equalities and, in particular, $m(b)$ and $m(c)$ are achieved (at $u$ and $v$ respectively), which due to Lemma \ref{L-subadd} (ii) yields
	\[
	m(a) \ge m(b) + m(c) + m(d) > m(\sqrt{b^2+c^2+d^2}) \ge m(a),
	\]
	a contradiction. If $d=0$, then $w_n \to 0$ in $L^2(\rn)$ and so $\int_{\rn}F_+(w_n)\,dx \to 0$ from (F1), (F2), and \eqref{e-GN}, which implies $\liminf_nJ(w_n) \ge 0$. Then there holds
	\[
	m(a) = \lim_nJ(u_n) = J(u) + J(v) + \lim_nJ(w_n) \ge m(b) + m(c) \ge m(\sqrt{b^2+c^2}) \ge m(a)
	\]
	and we reach a contradiction as before.
\end{proof}

\begin{Lem}\label{L-min}
	If (F0)--(F3), \eqref{e-etas}, and \eqref{e-etal} hold, then $\inf_{\cD}J$ is achieved.
\end{Lem}
\begin{proof}
	Let $u_n\in\cD$ such that $\lim_nJ(u_n)=\inf_{\cD}J$. In view of Lemmas \ref{L-cbb} and \ref{L-cptmin} there exists $u\in\cD$ such that, up to a subsequence and translations, $u_n\rightharpoonup u$ in $\hrn$ and $u_n\to u$ in $L^p(\rn)$, $2<p<2^*$, and a.e. in $\rn$ as $n\to\infty$. In particular, $\lim_n \int_{\rn} F_+(u_n) \, dx = \int_{\rn} F(u)_+ \, dx$ due to (F1), (F2), and \cite[Theorem 1]{BrezisLieb}, whence
	\[
	\inf_{\cD}J=\lim_nJ(u_n)\ge J(u)\ge \inf_{\cD}J.\qedhere
	\]
\end{proof}

\begin{proof}[Proof of Theorem \ref{T-main1}]
	In view of Lemmas \ref{L-neg} and \ref{L-min} there exists $u\in\cD$ such that $J(u)=\inf_{\cD}J<0$. Then from \cite[Proposition A.1]{MedSc} there exists $\lambda\ge0$ such that
	\begin{equation*}
		-\Delta u+\lambda u=F'(u).
	\end{equation*}
	Note that $\lambda=0$ if $u\in\mathring{\cD}$. Assume by contradiction that $\lambda=0$. Since $M=1$, it is standard that $u\in W_\textup{loc}^{2,p}(\rn)$ for every $p<\infty$, thus $u$ satisfies the Poho\v{z}aev identity \cite{BerLionsI,Jeanjean97}
	\begin{equation}\label{e-Poho}
		(N-2)\int_{\rn}|\nabla u|^2\,dx=2N\int_{\rn}F(u)\,dx
	\end{equation}
	and so $0>J(u)=|\nabla u|_2^2/N$, which is impossible, hence $(\lambda,u)$ solves \eqref{e-main}. Concerning the final part, we argue as in \cite[Proof of Theorem 1.4]{JeanLu2} if $N\ge2$ or $F'$ is locally Lipschitz continuous, otherwise the result follows from the properties of the Schwarz rearrangements \cite{LiebLoss}.
\end{proof}

\begin{proof}[Proof of Proposition \ref{P-decreasing}]
	Since $u$ is radial, we can argue as in \cite[Proof of Lemma 1]{BerLionsI} and have that $u \in \cC^2(\rn)$, hence the first part follows from the strong maximum principle \cite[Theorem VI.IV.4]{Evans}. As for the second part, assume by contradiction that $u$ is constant in the annulus $A:=\{r_1<|x|<r_2\}$ for some $r_2>r_1>0$. Let us consider the case $u>0$ (the one $u<0$ is analogous). Then $0=-\Delta u=F'(u)-\lambda u$ in $A$ and so $-\Delta u\le0$ in $\Omega:=\{|x|>r_1\}$ because $u$ is radially nonincreasing. At the same time $u$ attains its maximum over $\overline\Omega$ at every point of $A$. Suppose first that $N\ge2$, so that $\overline{\Omega}$ is connected. Then, again from the strong maximum principle, $u|_\Omega$ is constant, which is a contradiction. If $N=1$, it suffices to argue replacing $\overline{\Omega}$ with $\{x\ge r_1\}$, which is connected.
\end{proof}

\begin{proof}[Proof of Proposition \ref{P-main}]
	Since $F$ is even, every critical point of $J|_{\cD\cap X}$ is in fact a critical point of $J|_{\cD}$ in virtue of the principle of symmetric criticality \cite{Palais}. Then the proof is the same as that of Theorem \ref{T-main1}. Concerning the strict inequality $\inf_{\cD \cap X}J > \inf_{\cD}J$, it follows from the second part of Theorem \ref{T-main1}.
\end{proof}

\section{Ground states for $M\ge2$}\label{M2}

Let us recall (cf. \cite{LiebLoss}) that for a Borel set $A\subset\rn$ and for a Borel function $u\colon\rn\to\R$ that vanishes at infinity (i.e., $|\{|u|>t\}|<\infty$ for every $t>0$) we define by $A^*$ the Schwarz rearrangement of $A$ and by $u^*$ the Schwarz rearrangement of $u$. If, instead, $u\colon\rn\to\R^M$, then we set $u^* := (u_1^*,\dots,u_M^*)$. The following abstract lemma allows the coupling term of \eqref{e-main} to be rather generic.

\begin{Lem}\label{L-Schwarz}
	Let $\overline M\ge2$ be an integer. For every every $j \in \{ 1, \dots, \overline M \}$ let $\widetilde{F}_j\colon[0,\infty[\to[0,\infty[$ nondecreasing and define $\widetilde{F}\colon\R^{\overline{M}}\to\R$ as
	\[
	\widetilde{F}(u)=\prod_{j=1}^{\overline{M}}\widetilde{F}_j(|u_j|).
	\]
	For every $j \in \{1,\dots,\overline{M}\}$ assume that $\widetilde{F}_j$ is increasing or that it is continuous with $\widetilde{F}_j(0)=0$. Then for every Borel function $u\colon\rn\to\R^{\overline{M}}$ that vanishes at infinity there holds
	\[
	\int_{\rn}\widetilde{F}(u)\,dx\le\int_{\rn}\widetilde{F}(u^*)\,dx.
	\]
\end{Lem}
\begin{proof}
	Assume first that each $\widetilde{F}_j$ is increasing. From \cite[Theorem 1.13]{LiebLoss} we have
	\[\begin{split}
		\int_{\rn}\prod_{j=1}^{\overline{M}}\widetilde{F}_j(|u_j|)\,dx & =\int_{\rn}\prod_{j=1}^{\overline{M}}\int_0^\infty\chi_{\{\widetilde{F}_j(|u_j|)>t_j\}}(x)\,dt_j\,dx\\
		& =\int_0^\infty\cdots\int_0^\infty\int_{\rn}\prod_{j=1}^{\overline{M}}\chi_{\{\widetilde{F}_j(|u_j|)>t_j\}}(x)\,dx\,dt_1\cdots \,dt_{\overline{M}}
	\end{split}\]
	and similarly
	\[\begin{split}
		\int_{\rn}\prod_{j=1}^{\overline{M}}\widetilde{F}_j(u_j^*)\,dx & =\int_0^\infty\cdots\int_0^\infty\int_{\rn}\prod_{j=1}^{\overline{M}}\chi_{\{\widetilde{F}_j(u_j^*)>t_j\}}(x)\,dx\,dt_1\cdots \,dt_{\overline{M}}\\
		& =\int_0^\infty\cdots\int_0^\infty\int_{\rn}\prod_{j=1}^{\overline{M}}\chi_{\{\widetilde{F}_j(|u_j|)>t_j\}^*}(x)\,dx\,dt_1\cdots \,dt_{\overline{M}}
	\end{split}\]
	because
	\[
	\{\widetilde{F}_j(u_j^*)>t_j\}=\{u_j^*>\widetilde{F}_j^{-1}(t_j)\}=\{|u_j|>\widetilde{F}_j^{-1}(t_j)\}^*=\{\widetilde{F}_j(|u_j|)>t_j\}^*,
	\]
	therefore it suffices to prove that for every Borel $A_1,\dots,A_{\overline{M}} \subset \rn$
	\[
	\int_{\rn}\prod_{j=1}^{\overline{M}}\chi_{A_j}(x)\,dx\le\int_{\rn}\prod_{j=1}^{\overline{M}}\chi_{A_j^*}(x)\,dx,
	\]
	i.e., $|\cap_{j=1}^{\overline{M}}A_j|\le|\cap_{j=1}^{\overline{M}}A_j^*|$.
	
	Up to relabelling the sets, we can assume that $|A_1|\le\cdots\le|A_{\overline{M}}|$, whence $A_1^*\subset\cdots\subset A_{\overline{M}}^*$ and so $|\cap_{j=1}^{\overline{M}}A_j|\le|A_1|=|A_1^*|=|\cap_{j=1}^{\overline{M}}A_j^*|$.
	
	Now assume that some $\widetilde{F}_j$'s are continuous with $\widetilde{F}_j(0)=0$. We want to prove that, as in the previous case, $\{\widetilde{F}_j(u_j^*)>t_j\}=\{\widetilde{F}_j(|u_j|)>t_j\}^*$. Since $\widetilde{F}_j \colon [0,\infty[ \to [0,\infty[$ is nondecreasing and $\widetilde{F}_j(0)=0$, following e.g. \cite[p. 10]{RaoRen} we can define the \textit{generalized inverse function} $\widetilde{F}_j^{-1}\colon[0,\infty[\to[0,\infty[$ as
	\[
	\widetilde{F}_j^{-1}(t):=\inf\{s>0:\widetilde{F}_j(s)>t\}.
	\]
	Then it suffices to prove that $\{\widetilde{F}_j(v)>t\}=\{v>\widetilde{F}_j^{-1}(t)\}$ for every $t>0$ and every measurable $v\colon\rn\to[0,\infty[$ that vanishes at infinity. Observe that $\widetilde{F}_j^{-1}(t)=\infty$ if $t\ge\widetilde{F}_j\bigl(v(x)\bigr)$ for every $x\in\rn$ and, in that case, both sets above equal the empty set, hence we can assume that $t<\mathrm{ess}\sup\widetilde{F}_j\circ v$.
	
	Let $x\in\rn$ such that $\widetilde{F}_j\bigl(v(x)\bigr)>t$. Since $\widetilde{F}_j$ is continuous and $\widetilde{F}_j(0)=0$, there exists $s>0$ such that $t<\widetilde{F}_j(s)<\widetilde{F}_j\bigl(v(x)\bigr)$. Then $\widetilde{F}_j^{-1}(t)\le s$ and, since $\widetilde{F}_j$ is nondecreasing, $s<v(x)$. Now let $x\in\rn$ such that $v(x)>\widetilde{F}_j^{-1}(t)$. From the properties of infima, there exists $s>0$ such that $\widetilde{F}_j(s)>t$ and $v(x)\ge s$. Since $\widetilde{F}_j$ is nondecreasing, $\widetilde{F}_j\bigl(v(x)\bigr)\ge\widetilde{F}_j(s)$.
\end{proof}

The next result is inspired from \cite{Gidas}.

\begin{Lem}\label{L-N5}
	Let $f\colon[0,\infty[\to[0,\infty[$ be a continuous function that satisfies (P) and such that $f(t)>0$ if $t>0$. Then the problem
	\begin{equation}\label{e-Gidas}
		\begin{cases}
			-\Delta u\ge f(u)\\
			u\ge0\\
			u\in\cC^2(\rn)\cap L^\infty(\rn)
		\end{cases}
	\end{equation}
	does not admit positive solutions.
\end{Lem}
\begin{proof}
	We can assume $q>1$. If $u$ is a solution to \eqref{e-Gidas}, then there exists $C=C(u)>0$ such that $f\bigl(u(x)\bigr)\ge Cu^q(x)$ for every $x\in\rn$. Then we argue as in \cite[Proof of Theorem 8.4]{QS}.
\end{proof}

\begin{proof}[Proof of Theorem \ref{T-main2} (a)]
	In virtue of Remarks \ref{R-cbb} and \ref{R-neg} let $u^n\in\cD$ such that $J(u^n) \to \inf_{\cD}J < 0$ and $u\in\cD$ such that, up to a subsequence, $u^n \rightharpoonup u$ in $\hrn^M$. From Lemma \ref{L-Schwarz}, up to replacing $u^n$ with its Schwarz rearrangement, we can assume that each $u^n$ -- and consequently $u$ -- is radial, nonnegative, and radially nonincreasing; in particular, $u^n \to u$ in $L^p(\rn)^M$ for every $2<p<2^*$ and so, arguing as in the proof of Lemma \ref{L-min}, $J(u) = \inf_{\cD}J$. Then, from \cite[Proposition A.1]{MedSc}, there exists $\lambda \in [0,\infty[^M$ such that $-\Delta u_j + \lambda_j u_j = \partial_jF(u)$ for every $j \in \{1,\dots,M\}$, hence $u \in W_\textup{loc}^{2,p}(\rn)^M$ for every $p<\infty$ owing to \cite[Theorem 2.3]{BrezisLiebV}. Then we can argue as in \cite[Proof of Lemma 1]{BerLionsI} to obtain that $u \in \cC^2(\rn)^M$. Now assume by contradiction that $\lambda_j=0$ for some $j \in \{1,\dots,M\}$. Up to changing the order, we can assume that $j=M$. Then $-\Delta u_M \ge F_M'(u_M) \ge 0$. This yields that $u_M=0$: if $N\le4$, then it follows from \cite[Lemma A.2]{Ikoma}; if each $F_j'|_{[0,\infty[}$ satisfies (P), then it follows from Lemma \ref{L-N5}. For every $j\in\{1,\dots,M\}$ define
	\[\begin{split}
		J_j(w) & = \int_{\rn} \frac12|\nabla w|^2-F_j(w) \, dx, \quad w\in\hrn,\\
		\cD(j) & = \Set{w\in\hrn|\int_{\rn} w^2 \, dx\le a_j^2}.
	\end{split}\]
	Of course, $-\Delta u_j+\lambda_ju_j=F_j'(u_j)$ and, from Lemmas \ref{L-cbb} and \ref{L-neg}, $-\infty < m_j := \inf_{\cD(j)}J_j<0$ for every $j\in\{1,\dots,M\}$. Moreover,
	\[
	J(u) = \sum_{j=1}^{M-1} J_j(u_j) \ge \sum_{j=1}^{M-1} m_j.
	\]
	Since $\widetilde{F}_{\ell,j}\ge0$ for every $\ell \in \{1,\dots,L\}$ and $j \in \{1,\dots,M\}$, $J(w) \le \sum_{j=1}^M J_j(w_j)$ for every $w = (w_1,\dots,w_M) \in \hrn^M$, thus
	\[
	\sum_{j=1}^M m_j \ge \inf_{\cD}J = J(u) \ge \sum_{j=1}^{M-1} m_j,
	\]
	whence $m_M\ge0$, a contradiction. Then for every $j \in \{1,\dots,M\}$ there holds $\lambda_j>0$ and, in particular, $|u_j|_2 = a_j$. That each $u_j$ is positive follows from the strong maximum principle \cite[Theorem VI.IV.4]{Evans}.
\end{proof}

If $u,v\colon\rn\to\R$ are two Borel functions that vanish at infinity, $t>0$, and $x\in\rn$, we define $A^*(u,v;t) := B_r$, where $r\ge0$ is chosen so that
\[
\left|B_r\right| = \left|\Set{ x\in\rn | \left|u(x)>t\right| }\right| + \left|\Set{ x\in\rn | \left|v(x)>t\right| }\right|,
\]
and
\[
\{u,v\}^* := \int_0^\infty\chi_{A^*(u,v;t)}(x) \, dt.
\]
Observe that $A^*(u,v;t) = A^*(v,u;t)$ and, consequently, $\{u,v\}^* = \{v,u\}^*$.

\begin{Lem}\label{L-Shibata}
	Let $u,v\colon\rn\to\R$ be as above.
	\begin{itemize}
		\item [(a)] $\{u,v\}^*$ is nonnegative, radial, nonincreasing, lower semicontinuous, and for every $t>0$ there holds $\Set{x\in\rn | \{u,v\}^*>t} = A^*(u,v;t)$.
		\item [(b)] If $\Phi\colon[0,\infty[\to[0,\infty[$ is nondecreasing, lower semicontinuous, continuous at $0$, and such that $\Phi(0)=0$, then
		\[
		\{\Phi(|u|),\Phi(|v|)\}^* = \Phi(\{u,v\}^*).
		\]
		\item [(c)] If $\Phi\colon[0,\infty[\to\R$ is monotone and such that $\Phi\circ|u|,\Phi\circ|v| \in L^1(\rn)$, then
		\[
		\int_{\rn} \Phi(\{u,v\}^*) \, dx = \int_{\rn} \Phi(|u|) + \Phi(|v|) \, dx.
		\]
		\item [(d)] If $u,v\in\hrn$, then $\{u,v\}^*\in\hrn$ and $|\nabla\{u,v\}^*|_2^2 \le |\nabla u|_2^2 + |\nabla v|_2^2$. If, moreover, $u,v\in\cC^1(\rn)\setminus\{0\}$ are radial, positive, and nonincreasing, then
		\[
		|\nabla\{u,v\}^*|_2^2 < |\nabla u|_2^2 + |\nabla v|_2^2.
		\]
		\item [(e)] Let $\overline{M}\ge2$ be an integer and $u_j,v_j\ge0$ be Borel functions that vanish at infinity, $j\in\{1,\dots,\overline{M}\}$. Then
		\[
		\int_{\rn} \prod_{j=1}^{\overline{M}} u_j + \prod_{j=1}^{\overline{M}} v_j \, dx \le \int_{\rn} \prod_{j=1}^{\overline{M}} \{u_j,v_j\}^* \, dx.
		\]
	\end{itemize}
\end{Lem}
\begin{proof}
	The proofs of points (a)--(d) can be found in \cite{Ikoma,Shibata_2017}. Concerning the proof of point (e), it is similar to \cite[Proof of Lemma A.1 (v)]{Ikoma} (given for $\overline{M}=2$), but we include it here for the reader's convenience. From \cite[Theorem 1.13]{LiebLoss} we have
	\[\begin{split}
		\int_{\rn} \prod_{j=1}^{\overline{M}} u_j + \prod_{j=1}^{\overline{M}} v_j \, dx & = \int_0^\infty \dots \int_0^\infty \left| \bigcap_{j=1}^{\overline{M}}\{u_j>s_j\} \right| + \left| \bigcap_{j=1}^{\overline{M}}\{v_j>s_j\} \right| \, ds_1 \dots ds_{\overline{M}},\\
		\int_{\rn} \prod_{j=1}^{\overline{M}} \{u_j,v_j\}^* \, dx & = \int_0^\infty \dots \int_0^\infty \left| \bigcap_{j=1}^{\overline{M}} A^*(u_j,v_j;s_j) \right| ds_1 \dots ds_{\overline{M}}.
	\end{split}\]
	If we set $C_j := \Set{x\in\rn | u_j(x) > s_j}$ and $D_j := \Set{x\in\rn | v_j(x) > s_j}$, then it suffices to prove that
	\[
	\left|\bigcap_{j=1}^{\overline{M}} C_j\right| + \left|\bigcap_{j=1}^{\overline{M}} D_j\right| \le \left|\bigcap_{j=1}^{\overline{M}} A^*(u_j,v_j;s_j)\right|,
	\]
	which is true because
	\[\begin{split}
		\left|\bigcap_{j=1}^{\overline{M}} A^*(u_j,v_j;s_j)\right| & = \min_{1\le j\le \overline{M}} |C_j| + |D_j| \ge \min_{1\le j\le \overline{M}} |C_j| + \min_{1\le j\le \overline{M}} |D_j|\\
		& \ge \left|\bigcap_{j=1}^{\overline{M}} C_j\right| + \left|\bigcap_{j=1}^{\overline{M}} D_j\right|.\qedhere
	\end{split}\]
\end{proof}

\begin{Lem}\label{L-cptmin2}
	If (F0)--(F3), \eqref{e-etas}, and \eqref{e-etal} hold and $F$ is as in the first part of Theorem \ref{T-main2} (b) (possibly with $\eta_0<\infty$), then every minimizing sequence of $J|_{\cD}$ is relatively compact in $L^p(\rn)^M$ up to translations, $2<p<2^*$.
\end{Lem}
\begin{proof}
	We begin as in the proof of Lemma \ref{L-cptmin}, so let $u^n\in\cD$ such that $\lim_nJ(u^n) = \inf_{\cD}J$. There exist $u \in \cD \setminus \{0\}$ and $y_n\in\rn$ such that, up to a subsequence, $v^n := u^n(\cdot - y_n) - u \rightharpoonup 0$ in $\hrn^M$, $v^n\to0$ a.e. in $\rn$, and $\lim_n J(u^n) - J(v^n) = J(u)$. If $\lim_n \max_{y\in\rn} \int_{B(y,1)} |v^n|^2 \, dx = 0$, then the statement holds true; otherwise there exist $v\in\hrn^M\setminus\{0\}$ and $z_n\in\rn$ such that, up to a subsequence, $w^n := v^n(\cdot-z_n) - v \rightharpoonup 0$ in $\hrn^M$, $w^n\to0$ a.e. in $\rn$, $\lim_n J(v^n) - J(w^n) = J(v)$, and $\lim_n |u_j^n|_2^2 - |w_j^n|_2^2 = |u_j|_2^2 + |v_j|_2^2$ for every $j\in\{1,\dots,M\}$. Denoting $b_j := |u_j|_2$ and $c_j := |v_j|_2$ we have $a_j^2-b_j^2-c_j^2 \ge \liminf_n |w_j^n|_2^2 =: d_j^2 \ge0$. If $d=(d_1,\dots,d_M)=0$, then $w^n\to0$ in $L^2(\rn)^M$ and so $\liminf_nJ(w^n)\ge0$, thus using Remark \ref{R-subadd} (ii)
	\[
	m(a) = \lim_nJ(u_n) = J(u) + J(v) + \lim_nJ(w_n) \ge m(b) + m(c) \ge m(\sqrt{b^2+c^2}) \ge m(a);
	\]
	otherwise, define $\tilde{w}_j^n := \frac{d_j}{|w_j^n|_2}w_j^n$ if $d_j\ne0$, $\tilde{w}_j^n=0$ if $d_j=0$, and let $\tilde{w}^n := (\tilde{w}_1^n,\dots,\tilde{w}_M^n) \in \cS(d)$, hence again $\lim_n J(w^n) - J(\tilde{w}^n) = 0$ and so
	\[\begin{split}
		m(a) & = \lim_nJ(u_n) = J(u) + J(v) + \lim_nJ(w_n) = J(u) + J(v) + \lim_nJ(\tilde{w}_n)\\
		& \ge m(b) + m(c) + m(d) \ge m(\sqrt{b^2+c^2+d^2}) \ge m(a).
	\end{split}\]
	In particular, in either case, $J(u) = m(b)$ and $J(v) = m(c)$. Here is where this proof no longer follows that of Lemma \ref{L-cptmin}. Since $|u_j|_2 = |u_j^*|_2$, $|v_j|_2 = |v_j^*|_2$, and, in view of Lemma \ref{L-Schwarz}, $m(b)\le J(u^*)\le J(u)$ and $m(c)\le J(v^*)\le J(v)$, we can assume that $u$ and $v$ are radial, nonnegative, and radially nondecreasing. Moreover, since $u$ and $v$ are solutions to the differential equation in \eqref{e-main}, from \cite[Theorem 2.3]{BrezisLiebV} they are of class $\cC^1$, while from \cite[Theorem VI.IV.4]{Evans} each of their components is either identically $0$ or positive.
	
	\textit{Case 1:} there exists $k\in\{1,\dots,M\}$ such that $u_k\ne0$ and $v_k\ne0$. From Lemma \ref{L-Shibata} there holds $|\{u_j,v_j\}^*|_2^2 = b_j^2 + c_j^2$ for every $j\in\{1,\dots,M\}$ and
	\[\begin{split}
		|\nabla\{u_k,v_k\}^*|_2^2 < |\nabla u_k|_2^2 + |\nabla v_k|_2^2 \quad & \text{and} \quad |\nabla\{u_j,v_j\}^*|_2^2 \le |\nabla u_j|_2^2 + |\nabla v_j|_2^2,\\
		\int_{\rn} F_j(\{u_j,v_j\}^*) \, dx & = \int_{\rn} F_j(u_j) + F_j(v_j),\\
		\int_{\rn} F_{i,j}(\{u_i,v_i\}^*) F_{j,i}(\{u_j,v_j\}^*) \, dx & \ge \int_{\rn} F_{i,j}(u_i) F_{j,i}(u_j) + F_{i,j}(v_i) F_{j,i}(v_j) \, dx,
	\end{split}\]
	hence $J(u) + J(v) > J\bigl((\{u_j,v_j\}^*)_{j=1}^M\bigr)$ and we get the contradiction
	\[
	m(a) = J(u) + J(v) + \lim_n J(w^n) > m(\sqrt{b^2+c^2}) + m(d) \ge m(a).
	\]
	
	\textit{Case 2:} for every $j\in\{1,\dots,M\}$ there holds $u_j=0$ or $v_j=0$. Let $h,k\in\{1,\dots,M\}$, $h\ne k$, such that $u_h\ne0$ and $v_k\ne0$. We still have
	\[
	|\nabla\{u_j,v_j\}^*|_2^2 \le |\nabla u_j|_2^2 + |\nabla v_j|_2^2 \quad \text{and} \quad \int_{\rn} F_j(\{u_j,v_j\}^*) \, dx = \int_{\rn} F_j(u_j) + F_j(v_j).
	\]
	Moreover,
	\[\begin{split}
		\int_{\rn} F_{h,k}(\{u_h,v_h\}^*) F_{k,h}(\{u_k,v_k\}^*) \, dx & = \int_{\rn} F_{h,k}(\{u_h,v_h\}^*) F_{k,h}(\{v_k,u_k\}^*) \, dx\\
		& \ge \int_{\rn} F_{h,k}(u_h) F_{k,h}(v_k) + F_{h,k}(0) F_{k,h}(0) \, dx\\
		& > 0
	\end{split}\]
	and
	\[\begin{split}
		\int_{\rn} F_{i,j}(\{u_i,v_i\}^*) F_{j,i}(\{u_j,v_j\}^*) \, dx \ge \int_{\rn} F_{i,j}(u_i) F_{j,i}(u_j) + F_{i,j}(v_i) F_{j,i}(v_j) \, dx\\
		=\begin{cases}
			\int_{\rn} F_{i,j}(u_i) F_{j,i}(u_j) \, dx \quad \text{if } v_i=0 \text{ or } v_j=0,\\
			\int_{\rn} F_{i,j}(v_i) F_{j,i}(v_j) \, dx \quad \text{if } u_i=0 \text{ or } u_j=0,
		\end{cases}
	\end{split}\]
	thus we have again $J(u) + J(v) > J\bigl((\{u_j,v_j\}^*)_{j=1}^M\bigr)$ and conclude as before.
\end{proof}

\begin{proof}[Proof of Theorem \ref{T-main2} (b)]
	Thanks to Remarks \ref{R-cbb} and \ref{R-neg} let $u^n \in \cD$ such that $J(u^n) \to \inf_{\cD}J < 0$ and $u \in \cD$ such that, up to a subsequence, $u^n \rightharpoonup u$ in $\hrn^M$. Using Lemma \ref{L-cptmin2} we can argue as in the proof of Lemma \ref{L-min} and obtain that $\inf_{\cD}J = \lim_n J(u^n) \ge J(u) \ge \inf_{\cD}J$, hence from \cite[Proposition A.1]{MedSc} there exists $\lambda \in [0,\infty[^M$ such that $-\Delta u_j +\lambda_j u_j = \partial_jF(u)$ and, arguing as in the proof of Theorem \ref{T-main1}, we have that $\max_{j=1,\dots,M}\lambda_j > 0$. Moreover, up to replacing $u$ with its Schwarz rearrangement, we can assume that $u$ is radial, nonnegative (in fact, positive owing to \cite[Theorem VI.IV.4]{Evans}), and radially nonincreasing. Now we prove that $u\in\cS$, hence define $\bar{a} := (|u_j|_2)_{j=1}^M$ and assume by contradiction that $\bar{a}_j < a_j$ for some $j \in \{1,\dots,M\}$. Then from Remark \ref{R-subadd} (ii)  and Proposition \ref{P-gsem} (i) we get
	\[
	m(\bar{a}) \le J(u) = m(a) \le m\bigl(\sqrt{(a-\bar{a})^2+\bar{a}^2}\bigr) \le m(a-\bar{a}) + m(\bar{a})
	\]
	and so $m(a-\bar{a}) \ge 0$, in contrast with Remark \ref{R-neg}. The final part is proved exactly as in the proof of Theorem \ref{T-main2} (a).
\end{proof}

\begin{Rem}
	Note that the assumption $\eta_0=\infty$ is used in the proof of Theorem \ref{T-main2} (b) only to ensure $m(a-\bar{a})<0$ (because otherwise we could not know whether \eqref{e-etal} holds with $a-\bar{a}$ instead of $a$), therefore a minimizer of $J|_{\cD}$ exists also with $\eta_0<\infty$ and \eqref{e-etal} satisfied.
\end{Rem}

\begin{Cor}\label{C-cpt}
	In the assumptions of the first part of Theorem \ref{T-main1} or of the first part of Theorem \ref{T-main2} (b), every minimizing sequence for $J|_{\cD}$ is relatively compact in $\hrn^M$ up to translations
\end{Cor}
\begin{proof}
	Assume first $M=1$ and let $u_n\in\cD$ be such that $\lim_nJ(u_n) = \inf_{\cD}J$ and $u\in\cD$ such that, up to a subsequence and translations, $u_n\rightharpoonup u$ in $\hrn$ and $u_n\to u$ in $L^p(\rn)$, $2<p<2^*$. From the proof of Lemma \ref{L-min} we have that $\lim_n|\nabla u_n|_2 = |\nabla u|_2$, while from the proof of Theorem \ref{T-main1} we have that $a = \lim_n|u_n|_2 \ge |u|_2 = a$. The case $M\ge2$ is analogous.
\end{proof}

\section{On the ground state energy map}\label{gsem}

\begin{Prop}\label{P-gsem}
	Assume that (F0) is verified. We have as follows.
	\begin{itemize}
		\item [(i)] $m\colon]0,\infty[^M\to\R\cup\{-\infty\}$ is nonincreasing, i.e., if $0<a_j\le\bar{a}_j$ for every $j\in\{1,\dots,K\}$, then $m(a)\ge m(\bar{a})$.
		\item [(ii)] If (F1) and (F2) are satisfied and $\eta_0=\infty$, then $\displaystyle\lim_{|a|\to0^+}m(a)=0$.
		\item [(iii)] If (F3) is satisfied and $\eta_\infty=0$, then $\displaystyle\lim_{\substack{a_j\to\infty \\ j=1,\dots,M}}m(a)=-\infty$; in particular, $\displaystyle\lim_{a\to\infty}m(a)=-\infty$ if $M=1$.
		\item [(iv)] In the assumptions of the first part of Theorem \ref{T-main1} or those of the first part of Theorem \ref{T-main2}, $m$ is decreasing, i.e., if $0<a_j\le\bar{a}_j$ for every $j\in\{1,\dots,M\}$ and $a_k<\bar{a}_k$ for some $k\in\{1,\dots,M\}$, then $m(a)>m(\bar{a})$.
	\end{itemize}
\end{Prop}
\begin{proof}
	\textit{(i)} It is obvious from the definition of $m$.
	
	\textit{(ii)} Note that \eqref{e-etas} holds for every sufficiently small $a$. Fix $\varepsilon>0$. Let $a^n\to0$ and $u^n\in\cD(a^n)$ such that $J(u^n)\le m(a^n)+\varepsilon$. In particular, $u^n\to0$ in $L^2(\rn)^M$ and $u^n\in\cD(\tilde{a})$ for some $\tilde{a}\in]0,\infty[^M$, therefore $u^n$ is bounded in $\hrn^M$ due to Lemma \ref{L-cbb} and Remark \ref{R-cbb}. This, together with (F1), (F2), and \eqref{e-GN}, yields $\int_{\rn} F_+(u^n) \, dx \to 0$, which implies $\liminf_nJ(u^n)=\liminf_n|\nabla u^n|_2^2/2\ge0$, whence, in view also of Lemma \ref{L-neg} and Remark \ref{R-neg}, $0\ge\limsup_nm(a^n)\ge\liminf_nm(a^n)\ge-\varepsilon$. Letting $\varepsilon\to0^+$ we conclude.
	
	\textit{(iii)} Note that \eqref{e-etal} holds for every sufficiently large $a$. Fix $a \in ]0,\infty[^M$, $u \in \cD(a) \cap L^\infty(\rn)^M \setminus \{0\}$, and note that $\int_{\rn}F(u)\,dx>0$ from (F3) provided $|u|_\infty$ is sufficiently small. For every $j\in\{1,\dots,M\}$ let $a_j^n\to\infty$ and denote $b_n:=\max_{j=1,\dots,M}a_j/a_j^n$ and $u^n(x):=u(b_n^{2/N}x)$. Then $\lim_nb_n=0$, $u^n\in\cD(a^n)$, and
	\[
	m(a^n)\le J(u^n)=\frac1{b_n^2}\left(b_n^{4/N}\int_{\rn}\frac12|\nabla u|^2\,dx-\int_{\rn}F(u)\,dx\right)\to-\infty.
	\]
	
	\textit{(iv)} Let $a,\bar{a} \in ]0,\infty[^M$ as in the statement. Clearly $m(a)\ge m(\bar{a})$ from item (i). If $m(a)=m(\bar{a})$, then there exists $u \in \cS(a) \subset \cD(\bar{a}) \setminus \cS(\bar{a})$ such that $J(u) = m(a) = m(\bar{a})$, which is impossible.
\end{proof}

\appendix
\section{Some explicit computations}\label{App}

\begin{Prop}
	Let $M\ge2$ and $\nu_j\ge0$, $j\in\{1,\dots,M\}$, and define
	\[
	F(u)=\sum_{j=1}^M\nu_j|u_j|^{2_\#}.
	\]
	Then $\limsup_{|u|\to\infty}F(u)/|u|^{2_\#}=\max_{j=1,\dots,M}\nu_j$.
\end{Prop}
\begin{proof}
	By taking $u_j=0$ for every $j\in\{1,\dots,M\}$ but one we easily obtain $\limsup_{|u|\to\infty}F(u)/|u|^{2_\#}\ge\max_{j=1,\dots,M}\nu_j$. Since $F(u)\le\max_{j=1,\dots,M}\nu_j\sum_{j=1}^M|u_j|^{2_\#}$, it suffices to prove that for every $u\in\R^M$
	\[
	\left(|u_1|^{2_\#}+\dots+|u_M|^{2_\#}\right)^{1/{2_\#}}\le\left(u_1^2+\dots+u_M^2\right)^{1/2}.
	\]
	To this aim, we will prove that the function $\varphi\colon]0,\infty[\to]0,\infty[$ is decreasing, where $\varphi(t):=(b_1^t+\dots+b_M^t)^{1/t}$ for some $b_1,\dots,b_M>0$. From
	\[
	\varphi'(t)=\varphi(t)\left(\frac{b_1^t\ln b_1+\dots+b_M^t\ln b_M}{t(b_1^t+\dots+b_M^t)}-\frac1{t^2}\ln(b_1^t+\dots+b_M^t)\right)
	\]
	we have that $\varphi'(t)<0$ is equivalent to
	\[
	b_1^t\ln(b_1^t+\dots+b_M^t)+\dots+b_M^t\ln(b_1^t+\dots+b_M^t)>b_1^t\ln b_1^t+\dots+b_M^t\ln b_M^t,
	\]
	which is true because $\ln$ is increasing.
\end{proof}

\begin{Prop}
	Let $r_1,r_2>1$ such that $r_1+r_2=2_\#$ and define
	\[
	F(u)=|u_1|^{r_1}|u_2|^{r_2}.
	\]
	Then $\limsup_{|u|\to\infty}F(u)/|u|^{2_\#}=\sqrt{r_1^{r_1}r_2^{r_2}/2_\#^{2_\#}}$.
\end{Prop}
\begin{proof}
	Observe that, for $u_2\ne0$, $F(u)/|u|^{2_\#}=\varphi(|u_1/u_2|)$, with $\varphi(t):=t^{r_1}/(t^2+1)^{2_\#}$. Since $\max\varphi=\varphi(\sqrt{r_1/r_2})=\sqrt{r_1^{r_1}r_2^{r_2}/2_\#^{2_\#}}$, there holds
	\[
	\sqrt{\frac{r_1^{r_1}r_2^{r_2}}{2_\#^{2_\#}}}\ge\limsup_{|u|\to\infty}\frac{F(u)}{|u|^{2_\#}}\ge\limsup_{\substack{u_1=\sqrt{r_1/r_2}u_2\\ |u_2|\to\infty}}\frac{F(u)}{|u|^{2_\#}}=\sqrt{\frac{r_1^{r_1}r_2^{r_2}}{2_\#^{2_\#}}}.\qedhere
	\]
\end{proof}

\section*{Acknowledgments}

The author was partly supported by the National Science Centre, Poland (Grant No. 2017/26/E/ST1/00817). The author would also like to thank Jaros{\l}aw Mederski and Panayotis Smyrnelis for interesting discussions and suggestions.

\end{document}